\documentclass[leqno,11pt]{amsart}
\usepackage{amsmath}
\usepackage{amsfonts}
\usepackage{amssymb}
\usepackage{amsthm}

\usepackage{mathrsfs}
\usepackage{dsfont}
\usepackage{txfonts}
\usepackage{mathabx}
\usepackage{mathtools}
\usepackage{graphicx}
\usepackage[all]{xy}
\usepackage{pgf,tikz}
\usepackage{subfigure}
\usepackage{tikz-3dplot}
\usetikzlibrary{patterns,arrows}

\usepackage[pdftex]{hyperref}

\usepackage{color}
\usepackage[shortlabels]{enumitem}

\newtheorem{theorem}{Theorem}[section]

\newtheorem{proposition}[theorem]{Proposition}

\theoremstyle{definition}

\newtheorem{notation}[theorem]{Notation}

\theoremstyle{remark}
\newtheorem{remark}[theorem]{Remark}

\newcommand{\bd}{\partial}
\newcommand{\rn}{\mathbb{R}^n}
\newcommand{\rno}{\mathbb{R}^{n+1}}

\newcommand{\rth}{\mathbb{R}^3}

\newcommand{\cL}{\mathcal{L}}

\title{Uniqueness of closed  self-similar solutions \\  to the Gauss curvature flow}

\author{Kyeongsu Choi}
\address{ {\bf Kyeongsu Choi:} Department of Mathematics, Columbia University, 2990 Broadway, New York, NY 10027, USA.}
\email{kschoi@math.columbia.edu}
\author{Panagiota Daskalopoulos}
\address{ {\bf P. Daskalopoulos:} Department of Mathematics, Columbia University, 2990 Broadway, New York, NY 10027, USA.}
\email{pdaskalo@math.columbia.edu}

\begin{document}

\maketitle

\begin{abstract}
We show the uniqueness of  strictly convex closed smooth  self-similar solutions   to the $\alpha$-Gauss curvature 
flow with $(1/n) < \alpha < 1+(1/n)$. We introduce a Pogorelov type computation, and then we apply the strong maximum principle. Our work combined with earlier works on the Gauss Curvature flow imply  that the $\alpha$-Gauss curvature 
flow with $(1/n) < \alpha < 1+(1/n)$ shrinks  a strictly convex closed smooth hypersurface to a round sphere. 
\end{abstract}

\section{introduction}

We recall that given $\alpha >0$, an one-parameter family of immersions $F:M^n \times [0,T) \to \rno$ is a solution of the $\alpha$-Gauss curvature flow, if for 
each $t \in [0,T)$, $F(M^n,t)=\Sigma_t$ is a complete convex hypersurface embedded in $\rno$, and $F(\cdot,t)$ satisfying
\begin{equation*}
\frac{\bd}{\bd t}  F(p,t)= K^\alpha(p,t) \vec{n}(p,t). 
\end{equation*} 
where $K(p,t)$ and $\vec{n}(p,t)$ are the Gauss curvature and the interior unit vector of $\Sigma_t$ at the point $F(p,t)$, respectively. In particular, if $\alpha=1$ we call  the immersion $F:M^n \times [0,T) \to \rno$ a solution of the Gauss curvature flow.

We consider a closed strictly convex smooth self-similar solution to  the $\alpha$-Gauss curvature flow for $\alpha \in (\frac{1}{n},1+\frac{1}{n})$. Since a closed self-similar solution $\Sigma$ is a shrinking solution, there exists an immersion $F:M^n \to \rno$ such that $F(M^n)=\Sigma$ and the following holds
\begin{align*}\label{eq:INT Shrinker}
K^{\alpha}(p)=-\langle F(p), \vec{n}(p) \rangle. \tag{$*^\alpha$}
\end{align*}

\bigskip In \cite{Fir74GCF} W. Firey introduced the Gauss curvature flow $\alpha=1$   and 
showed  (assuming the existence and regularity of the flow) 
 that a convex closed and centrally symmetric solution 
in $\rth$ contracts to a point and becomes   a round sphere after rescaling.  He also conjectured that
the same result holds true without the symmetry assumption.

\smallskip
In \cite{Tso85GCF} K. Tso 
established the existence and uniqueness of the Gauss curvature flow $\alpha=1$ in $\rno$  
and showed that the flow contracts a closed, smooth and strictly convex hypersurface to a point in finite time.   In \cite{Chow85GCF} 
B. Chow extended Tso's   result  to the $\alpha$-Gauss curvature flow for all $\alpha >0$ in $\rno$.

\smallskip

In \cite{Calabi72Affine} E. Calabi showed that if $\alpha = \frac{1}{n+2}$, closed self-similar solutions are ellipsoids. On the other hand, B. Chow proved in \cite{Chow85GCF} that if $\alpha=\frac{1}{n}$, a strictly convex closed solution converges to a round sphere after  normalizing the enclosed volume, which implies that the strictly convex closed self-similar solution is the unit sphere. 
\smallskip

In  \cite{A99GCF}  B. Andrews proved Firey's conjecture, showing that the Gauss curvature flow 
$\alpha=1$ and $n=2$ contracts a weakly  convex hypersurface in $\rth$ to  
a round sphere. Also, in \cite{AC11aGCF} B. Andrews and X. Chen 
established the same convergence result for  $\alpha \in (\frac{1}{2},1)$ and $n=2$.  The proof 
of B. Andrew's result in \cite{A99GCF}  is based on a beautiful pinching estimate which unfortunately does not generalize in higher dimensions. 

\smallskip 
 
Recently,  P. Guan and L. Ni  \cite{NG13GCFEntropy} obtained the convergence of a centrally symmetric solution of  
the Gauss curvature flow $\alpha=1$ to a sphere 
and  in \cite{ANG15aGCF} they extended  the same result  to $\alpha \geq 1$ jointly with B. Andrews. 
The convergence of the Gauss curvature flow $\alpha=1$ to the sphere without any symmetry assumption in higher dimensions  has remained an open question. 

\smallskip
On the other hand it follows from the works  \cite{A00aGCF, ANG15aGCF, NG13GCFEntropy, KL13aGCF}, 
that  if $\alpha > \frac{1}{n+2}$, then a strictly convex closed solution to the $\alpha$-Gauss curvature flow 
converges to a strictly convex smooth closed self-similar solution after normalizing the enclosed volume.
Thus  the convergence of the $\alpha$-Gauss curvature flow to the sphere for $\alpha > \frac{1}{n+2}$
is reduced to the classification of convex smooth closed self-similar solutions. 

\smallskip
In this work we show that if $\alpha  \in  (\frac{1}{n},1+\frac{1}{n})$ then the only strictly convex smooth and closed self-similar solution
of the $\alpha$-Gauss curvature flow is the round sphere. 

\smallskip 

\begin{theorem}[Uniqueness of closed self-similar solutions]\label{thm:INT Uniqueness}
Given $\alpha\in  (\frac{1}{n},1+\frac{1}{n})$, the unit $n$-sphere is the unique closed strictly convex smooth solution to {\em \eqref{eq:INT Shrinker}}. 
\end{theorem}

As we discussed above, 
the results in \cite{A00aGCF, ANG15aGCF, NG13GCFEntropy, KL13aGCF} combined with Theorem \ref{thm:INT Uniqueness} imply the convergence of the $\alpha$-Gauss curvature flow to the round sphere, which in particular
proves the higher dimensional  Firey's conjecture.

\begin{theorem} Let $\Sigma_t$ be a strictly convex, closed and smooth solution  to the $\alpha$-Gauss curvature flow with 
$\alpha \in  (\frac{1}{n},1+\frac{1}{n})$, $n \geq 2$. Then, there exists a finite time $T$ at which the surface $\Sigma_t$ converges after rescaling to the round sphere. 

\end{theorem}

\medskip
\textit{Discussion of the proof :} In \cite{Chow85GCF}, B. Chow used the quantity  $HK^{-\frac{1}{n}}$ as a subsolution to obtain the convergence of 
the $\alpha$-Gauss curvature flow to the sphere when $\alpha=\frac 1n$. The third order terms of the evolution equation of $HK^{-\frac{1}{n}}$ are controlled by the concavity of the $K^{\frac{1}{n}}$ operator. Also, the evolution equation has no reaction term, because $HK^{-\frac{1}{n}}$ is a homogeneous of degree $0$ function. 

\smallskip
In this paper, we use the quantity $\displaystyle{w(p) \coloneqq K^\alpha \lambda_{\min}^{-1}(p) - \frac{n\alpha -1}{2n\alpha}|F|^2(p)}$, where $\lambda_{\min}$ 
is the smallest principal curvature.  The second order terms in the equation of $\cL\, (K^\alpha \lambda_{\min}^{-1})$ can be controlled by  terms that appear in the equation of  $\frac{n\alpha -1}{2n\alpha}\cL \,|F|^2$, where $\cL$ is the linearized elliptic operator given in Notation \ref{not:Pre notation}. Hence, we only need to control the third order terms of the equation of $\cL \, w$. To deal with the third order terms, we adopt a Pogorelov type estimate with $\lambda_{\min}^{-1}$ replaced by  $(b^{1i}g_{ij}b^{j1})^{\frac{1}{2}}$, 
where $\{b^{ij}\}$ is the inverse matrix of $\{h_{ij}\}$ and at a point where $\lambda_{\min} = b^{11}.$ 
This is the main calculation in our work and will be done in the proof of Theorem \ref{thm:Pog Pogrelov estimate},
where we will show that if $w(p)$ attains its maximum at a point $F(p_0)$, then  the point $F(p_0)$ is an {\em umbilical point}.

\smallskip 
In section \ref{sec-Strong maximum principle} we will use the strong maximum principle to establish our uniqueness 
result, Theorem \ref{thm:INT Uniqueness}. To this end, we need to  introduce the quantity 
$\displaystyle{f(p)\coloneqq K^\alpha \sum^n_{i=1}\lambda_{i}^{-1}(p) - \frac{n\alpha -1}{2\alpha}|F|^2(p)}$ and first show in Proposition \ref{prop:Smax Symmetric function} that 
if it attains its maximum at a point $F(p_0)$,  then the point $F(p_0)$  is also an umbilical point (notice that $\lambda_1, \cdots,\lambda_n$ denote as usual the principal curvatures). This is an immediate consequence of Theorem \ref{thm:Pog Pogrelov estimate}.  Then, we will apply the strong maximum principle on  $f(p)$
and prove our uniqueness result.  In the Pogorelov type estimate on $w(p)$ we can diagonalize the  
second fundamental form $h_{ij}$  only at one given point (the maximum point). The reason we need to  use the quantity $f(p)$  is that in this case we can diagonalize 
$h_{ij}$ at each point. 

\begin{remark}[Pogorelov estimate on powers of a matrix]\label{rmk:INT Pogorelov on power matrix}
Pogorelov type estimates in our context  have been frequently  applied in the past  by using $b^{11}$, the first entry of a matrix $A^{-1} \coloneqq\{b^{ij}\}$. However, 
if one applies the Pogorelov estimate for $b^{11}K^\alpha -\frac{n\alpha-1}{2n\alpha}|F|^2$,  one  can obtain the result of Theorem \ref{thm:Pog Pogrelov estimate} only for $\alpha \in (\frac{1}{n},\frac{1}{2}]$. In this work, by using  instead  $(b^{1i}g_{ij}b^{j1})^{\frac{1}{2}}$, the root of the first entry of the square $A^{-2}$ of the matrix $A^{-1}$, we are able to extend the result of Theorem \ref{thm:Pog Pogrelov estimate} to the range of exponents $\alpha \in (\frac{1}{n},1+\frac{1}{n})$, which includes the classical  case 
of the Gauss curvature flow $\alpha=1$. 

One can apply a similar  Pogorelov type estimate using the $m$-th root of the first entry of $A^{-m}$,  with large $m \in \mathbb{N}$ (depending on $n$) and 
 extend our result to the range of exponents $\alpha \in (\frac{1}{n},1+(\frac{n-1}{n})^{\frac{1}{2}})$. Notice that if $\alpha = 1+(\frac{n-1}{n})^{\frac{1}{2}}$, then  we have $I_1 =0$, where $I_1=\frac{n\alpha-1}{n\alpha}+1-\alpha$ is given in the proof of Theorem \ref{thm:Pog Pogrelov estimate}. 

\smallskip 

Since our goal of this paper is to prove Firey's conjecture in higher dimensions, we provide the proof of  the uniqueness of closed self-shrinkers to the $\alpha$-Gauss curvature flow for $\alpha \in (\frac{1}{n},1+\frac{1}{n})$ by using $A^{-2}$. 
\end{remark}

\section{Preliminaries}

 \begin{notation}\label{not:Pre notation}
For reader's convenience, we summarize the notation as follows. 

\begin{enumerate}
\item We recall the metric $g_{ij} = \langle F_i, F_j \rangle$, where $F_i \coloneqq \nabla_i F$, and  its inverse matrix  $g^{ij}$ of $g_{ij}$, namely $g^{ij}g_{jk}=\delta^i_k$. Also, we use the notation $F^i=g^{ij} \, F_j$.

\item For a strictly convex smooth hypersurface $\Sigma$, we denote by   $b^{ij}$ inverse matrix of its  {\em second fundamental form}  $h_{ij}$, namely $b^{ij}h_{jk}=\delta^i_k$. 

\item We denote by $\cL$  the {\em linearized }   operator 
$$\cL =\alpha  K^{\alpha} b^{ij}\nabla_i \nabla_j$$
Also, $\langle \;, \;\rangle_\cL$ denotes the associated inner product 
$\displaystyle \langle \nabla f,\nabla g \rangle_\cL :=  \alpha  K^{\alpha} b^{ij}\nabla_i f \nabla_j g$, 
 where $f,g$ are differentiable functions on $M^n$, and $\|\cdot\|_\cL $ denotes the $\cL$-norm given by the inner product $\langle \;, \;\rangle_\cL$.

\item We denote as usual by $H$ and $\lambda_{\min}$  the {\em  mean curvature} and the {\em  smallest principal curvature}, respectively.

\item We will use in the sequel the functions  $f:M^n \to \mathbb{R}$ and $w:M^n \to \mathbb{R}$ defined  by
\begin{align*}
& f(p)=\Big(K^{\alpha}b^{ij}g_{ij}- \frac{n\alpha-1}{2\alpha}|F|^2\Big)(p), && w(p)=\Big(K^{\alpha}\lambda_{\min}^{-1}- \frac{n\alpha-1}{2n\alpha}|F|^2\Big)(p).
\end{align*}
\end{enumerate}
 
 \end{notation}

\bigskip

\begin{proposition}
Given a strictly convex smooth solution $F:M^n \to \rno$ of \eqref{eq:INT Shrinker}, the following hold
\begin{align*}
\nabla_i b^{jk} =& -b^{jl}b^{km}\nabla_i h_{lm} , \label{eq:Pre Db} \tag{2.1}\\
\cL \, |F|^2 =& 2\alpha K^{\alpha} b^{ij}(g_{ij}-h_{ij}K^{\alpha})=2\alpha K^{\alpha} b^{ij}g_{ij} -2n\alpha K^{2\alpha}, \label{eq:Pre L|F|^2} \tag{2.2}\\
\nabla_i K^{\alpha}=&h_{ij}\langle F, F^j \rangle , \label{eq:Pre DK^a} \tag{2.3}\\
\cL \, K^\alpha =&\langle F,\nabla K^\alpha\rangle+n\alpha K^\alpha -\alpha K^{2\alpha}H , \label{eq:Pre LK^a} \tag{2.4}\\
 \cL \, b^{pq}=&  K^{-\alpha}b^{pr}b^{qs}\nabla_r K^\alpha \nabla_s K^\alpha+\alpha K^\alpha b^{pr}b^{qs}b^{ij}b^{km}\nabla_r h_{ik}\nabla_s h_{jm} \label{eq:Pre Lb} \tag{2.5}\\
 &+\langle F,\nabla  b^{pq} \rangle - b^{pq} -(n\alpha-1)g^{pq} K^\alpha+\alpha K^\alpha H b^{pq}.
\end{align*}
\end{proposition}

\begin{proof} From $\nabla_i (b^{jk}h_{kl})=\nabla_i \delta^j_l =0$, we can derive $h_{kl}\nabla_i b^{jk}=-b^{jk} \nabla_ih_{kl}$. Hence, we have \eqref{eq:Pre Db} by
\begin{align*}
\nabla_i b^{jm}= b^{lm}h_{kl}\nabla_i b^{jk}=-b^{lm}b^{jk} \nabla_ih_{kl}.
\end{align*}
Also, by definition $\cL \coloneqq \alpha K^\alpha b^{ij}\nabla_i \nabla_j$ we have
\begin{align*}
\cL \, |F|^2 = 2\alpha K^\alpha b^{ij}\langle F_i,F_j \rangle+2\alpha K^\alpha b^{ij}\langle F,\nabla_i\nabla_j F\vec{n}\rangle=2\alpha K^\alpha b^{ij}g_{ij}+2\alpha K^\alpha b^{ij}\langle F,h_{ij}\vec{n}\rangle.
\end{align*}
Thus, the given equation \eqref{eq:INT Shrinker} implies \eqref{eq:Pog D|F|^2}.

Equation \eqref{eq:Pre DK^a} can be simply obtained by differentiating \eqref{eq:INT Shrinker} 
\begin{align*}
\nabla_i K^{\alpha}=h_{ik}\langle F, F^k \rangle.
\end{align*}
Differentiating   the equation above again we  obtain 	
\begin{align*}
\nabla_i\nabla_j K^{\alpha}=\nabla_i h_{jk}\langle F, F^k \rangle+h_{ij}+h_{ik}h^k_j\langle F, \vec{n} \rangle=\langle F,\nabla h_{ij} \rangle+h_{ij}-h_{ik}h^k_jK^\alpha.
\end{align*}
On the other hand, \eqref{eq:Pre Db} and direct differentiation yield
\begin{align*}
\nabla_i\nabla_j K^{\alpha}=\nabla_i(\alpha K^{\alpha}b^{pq}\nabla_j h_{pq})=\alpha K^\alpha b^{pq}\nabla_i\nabla_j h_{pq}+ \alpha^2K^{\alpha}b^{rs}b^{pq}\nabla_ih_{rs} \nabla_j h_{pq}-\alpha K^\alpha b^{pr}b^{qs}\nabla_i h_{rs}\nabla_j h_{pq}.
\end{align*}
Observing 
\begin{align*}
\nabla_i \nabla_j h_{pq}=\nabla_i \nabla_p h_{jq}=&\nabla_p \nabla_i h_{jq}+R_{ipjm}h^m_q+R_{ipqm}h^m_j \\
=&\nabla_p \nabla_q h_{ij}+(h_{ij}h_{pm}-h_{im}h_{jp})h^m_q+(h_{iq}h_{pm}-h_{im}h_{pq})h^m_j
\end{align*}
we obtain
\begin{align*}
\alpha K^{\alpha} b^{pq}\nabla_i \nabla_j h_{pq}=\alpha K^{\alpha}b^{pq}\nabla_p \nabla_q h_{ij}+\alpha K^{\alpha} H h_{ij}-n\alpha K^{\alpha} h_{im}h^m_j=\cL \, h_{ij}+\alpha K^{\alpha} H h_{ij}-n\alpha K^{\alpha} h_{im}h^m_j.
\end{align*} 
Combining the equations above yields
\begin{align*}\label{eq:Pre Lh}
\cL \, h_{ij} = &-\alpha^2K^{\alpha}b^{rs}b^{pq}\nabla_ih_{rs} \nabla_j h_{pq}+\alpha K^\alpha b^{pr}b^{qs}\nabla_i h_{rs}\nabla_j h_{pq}\tag{2.6}\\
&+\langle F,\nabla h_{ij} \rangle+h_{ij}+(n\alpha-1)h_{ik}h^k_jK^\alpha -\alpha K^{\alpha} H h_{ij}.
\end{align*} 
We now observe
\begin{align*}
\cL \, K^\alpha=&\alpha K^\alpha b^{ij} \nabla_i ( \alpha K^\alpha b^{pq} \nabla_j h_{pq})\\
=&\alpha^3 K^{2\alpha}b^{ij}b^{pq}b^{rs}\nabla_i h_{rs}\nabla_j h_{pq}-\alpha^2 K^{2\alpha}b^{ij}b^{pr}b^{qs}\nabla h_{rs}\nabla_j h_{pq}+\alpha K^\alpha b^{pq}\cL\, h_{pq}
\end{align*}
which gives  \eqref{eq:Pre LK^a}, since 
\begin{align*}
\cL \, K^\alpha =\alpha K^{\alpha}b^{ij}\big(\langle F,\nabla h_{ij} \rangle+h_{ij}+(n\alpha-1)h_{ik}h^k_jK^\alpha -\alpha K^{\alpha} H h_{ij}\big)=\langle F,\nabla K^\alpha\rangle+n\alpha K^\alpha -\alpha K^{2\alpha}H .
\end{align*} 

Finally, by using \eqref{eq:Pre Db}, we can derive
 \begin{align*}
 \cL \, b^{pq}= \alpha K^\alpha b^{ij} \nabla_i ( -b^{pr}b^{qs}\nabla_j h_{rs} )=2\alpha K^\alpha b^{ij}b^{pk}b^{rm}b^{qs}\nabla_i h_{km}\nabla_j h_{rs}-b^{pr}b^{qs} \cL \, h_{rs}.
 \end{align*}
 Applying \eqref{eq:Pre Lh} yields
 \begin{align*}
 \cL \, b^{pq}=& \alpha^2 K^{\alpha}b^{pr}b^{qs}b^{ij}b^{km}\nabla_r h_{ij} \nabla_s h_{km}+\alpha K^\alpha b^{pr}b^{qs}b^{ij}b^{km}\nabla_r h_{ik}\nabla_s h_{jm}\\
 &+\langle F,\nabla  b^{pq} \rangle - b^{pq} -(n\alpha-1)g^{pq} K^\alpha+\alpha K^\alpha H b^{pq}.
 \end{align*}
Thus, $\nabla K^\alpha=\alpha K^\alpha b^{ij}\nabla h_{ij}$ gives the desired result.
\end{proof}

\section{Pogorelov type computation}\label{sec-Pogorelov}

We consider the function $w:M^n \to \mathbb{R}$ given by 
$$w(p) \coloneqq \big(K^{\alpha}\lambda_{\min}^{-1}- \frac{n\alpha-1}{2n\alpha}|F|^2\big)(p).$$ We will employ in this
section a Pogorelov type computation to show that the maximum point of $w(p)$ is an umbilical point.
We begin with the following standard  observation which we include here for the reader's convenience. 

\begin{proposition}[Euler's formula] \label{prop:Pog Euler's formula} Let $\Sigma$ be a smooth strictly convex hypersurface  and $F:M^n \rightarrow \rno$ be a smooth immersion with $F(M^n)=\Sigma$. Then, given a coordinate chart $\varphi:U(\subset \rn) \to M^n$ of a point $p \in \varphi(U) $, the following holds  for each $i\in \{1,\cdots,n\}$
\begin{align*}
\sum_{j=1}^n\frac{b^{ij}b_j^i(p)}{g^{ii}(p)} \leq \frac{1}{\lambda_{\min}^2(p)}. 
\end{align*}
\end{proposition}

\begin{proof}
For a fixed point $p\in M^n$, we choose an  orthonormal basis  $\{ E_1, \cdots,E_n \}$ of $T\Sigma_{F(p)}$ such that $L(E_j) =\lambda_j\, E_j $,
where $L$ is the Weingarten map  and $\lambda_1,\cdots,\lambda_n$ are the principal curvatures of $\Sigma$ at $p$. 
Given a chart $(\varphi,U)$ of $p \in \varphi(U) \subset M^n$, we denote by $\{ a_{ij}\}$  the matrix satisfying $F_i(p)\coloneqq \nabla_i F(p)  =a_{ij}E_j$  and by $\{c_{ij}\}$  the diagonal matrix  $\text{diag}(\lambda_1,\cdots,\lambda_n)$. We also denote by $\{a^{ij}\}$ and $\{c^{ij}\}$ the inverse matrices of $\{a_{ij}\}$ and $\{c_{ij}\}$, respectively. 

We observe $g_{ij}(p)=\langle F_i,F_j \rangle(p)=\langle a_{ik}E_k,a_{jl}E_l \rangle =a_{ik}a_{jk}$. Also, we can obtain $F^i(p)=a^{ji}E_j$ by $a^{ji}=a^{jk}\langle F_k(p),F^i(p)\rangle=\langle a^{jk}a_{kl}E_l,F^i(p)\rangle=\langle E_j,F^i(p) \rangle$. So, we have $g^{ij}(p)=\langle F^i,F^j\rangle(p)= a^{ki}a^{kj}$. In addition, $   L F_i(p)=h_{ij}(p) F^j(p)  $  implies
\begin{align*}
a^{mi}E_m=&F^i(p)=b^{ij}(p)h_{jk}(p)F^k(p)=b^{ij}(p)LF_j(p)\\=&b^{ij}(p)L(a_{jk}E_k)
=b^{ij}(p)a_{jk}L(E_k)=b^{ij}(p)a_{jk}\lambda_kE_k=b^{ij}(p)a_{jk}c_{km}E_m. 
\end{align*}
Hence, we have $b^{in}(p)=b^{ij}(p)a_{jk}c_{km}c^{ml}a^{ln}=a^{mi}c^{ml}a^{ln}$, and thus the following holds
\begin{align*}
b^{1r}g_{rs}b^{s1}(p)=&a^{i1}c^{ij}a^{jr}a_{rk}a_{sk}a^{m1}c^{ml}a^{ls}=a^{i1}c^{ij}\delta^j_k \delta^l_k a^{m1}c^{ml}=a^{i1}c^{ij} a^{m1}c^{mj}=\sum_{j} (a^{j1})^2\lambda^{-2}_j \\
\leq  & \sum_{j} (a^{j1})^2\lambda^{-2}_{\min} = \lambda^{-2}_{\min}\sum_{k,j} \langle a^{k1}E_k,a^{j1}E_j \rangle = \lambda^{-2}_{\min} \langle F^1(p),F^1(p) \rangle = \lambda^{-2}_{\min}g^{11}(p),
\end{align*}
which is the desired result for $i=1$ and  we can obtain the same result for each $i \in \{1,\cdots,n\}$.
\end{proof}

\bigskip

We will now show that one of  the Pogorelov type expressions  of the function $w$ plays a role as a subsolution of \eqref{eq:INT Shrinker} at a given maximum point, to  imply that the maximum point of $w(p)$ is an umbilical point.

\begin{theorem}[Pogorelov type computation]\label{thm:Pog Pogrelov estimate}
Let $\Sigma$ be a strictly convex smooth closed solution of \eqref{eq:INT Shrinker} for an exponent $\alpha \in (\frac{1}{n},1+\frac{1}{n})$. Assume that $F:M^n \to \rno$ is a smooth immersion such that $F(M^n)= \Sigma$, and the continuous function $w(p)$ attains its maximum at a point $p_0$. Then, $F(p_0)$ is an umbilical point and $\nabla |F|^2(p_0)=0$ holds.
\end{theorem}

\begin{proof}
We begin by choosing a coordinate chart $(U,\varphi)$ of $p_0 \in \varphi(U) \subset M^n$ such that  the covariant derivatives $\big\{ \nabla_i F(p_0)\coloneqq\bd_i (F \circ \varphi)(\varphi^{-1}(p_0)) \big \}_{i=1,\cdots,n}$ form an orthonormal basis of $T\Sigma_{F(p_0)}$ satisfying
\begin{align*}
g_{ij}(p_0)=\delta_{ij}, \qquad  h_{ij}(p_0)=\delta_{ij}\lambda_i(p_0), \qquad  \lambda_1(p_0) =\lambda_{\min}(p_0),
\end{align*}
which guarantees   $b^{11}(p_0) = \lambda^{-1}_{\min}(p_0) $ and $g^{11}(p_0) = 1$.
Next, we  define the  function $\widebar w: \varphi(U) \rightarrow \mathbb{R}$ by
\begin{align*}
\widebar w(p)\coloneqq K^\alpha \Big(\frac{b^{1i}g_{ij}b^{j1}}{g^{11}} \Big)^{\frac{1}{2}}(p) -\frac{n\alpha-1}{2n\alpha}|F|^2(p). 
\end{align*}
Then,  by Proposition \ref{prop:Pog Euler's formula} we have 
\begin{align*}
\widebar w(p) \leq w(p) \leq w(p_0)  =\widebar w(p_0),
\end{align*}
which means that $\widebar w$ attains its maximum at $p_0$.  

\smallskip

We will now calculate  $\displaystyle{\cL \widebar w \coloneqq \alpha K^\alpha b^{ij} \nabla_i \nabla_j \widebar w} \, $  at the point $p_0$.
First  we derive the following equation from \eqref{eq:Pre Lb} 
\begin{align*}
 \cL \, \big ( b^1_pb^{p1} \big ) =&2\alpha K^{\alpha}b^{ij}\nabla_i b^{p1} \nabla_j b^1_p+ 2 K^{-\alpha}b^1_p b^{pr}b^{1s}\nabla_r K^{\alpha} \nabla_s K^{\alpha}+2\alpha K^\alpha b^1_pb^{pr}b^{1s}b^{ij}b^{km}\nabla_r h_{ik}\nabla_s h_{jm}\\
 &+\langle F,\nabla ( b^1_pb^{p1}) \rangle - 2b^1_pb^{p1} -2(n\alpha-1) K^\alpha b^{11}+2\alpha K^\alpha H b^1_pb^{p1}.
 \end{align*}
Thus, we obtain
 \begin{align*}
 \cL \, \Big(\frac{b^1_pb^{p1}}{g^{11}} \Big)^{\frac{1}{2}}=&-\frac{\alpha K^\alpha b^{ij}\nabla_i (b^1_pb^{1p})\nabla_j (b^1_qb^{1q})}{4(b^1_rb^{r1})^{\frac{3}{2}}(g^{11})^{\frac{1}{2}}}+\frac{\alpha K^{\alpha}b^{ij}\nabla_i b^{p1} \nabla_j b^1_p}{(b^1_qb^{q1}g^{11})^{\frac{1}{2}}}\\
 &+\frac{b^1_p b^{pr}b^{1s} \nabla_r K^{\alpha} \nabla_s K^{\alpha}}{K^\alpha(b^1_qb^{q1}g^{11})^{\frac{1}{2}}}+\frac{\alpha K^\alpha b^1_pb^{pr}b^{1s}b^{ij}b^{km}\nabla_r h_{ik}\nabla_s h_{jm}}{(b^1_qb^{q1}g^{11})^{\frac{1}{2}}}\\
 &+\big \langle F,\nabla   \big(b^1_pb^{p1}/g^{11} \big)^{\frac{1}{2}} \big \rangle -  \Big(\frac{b^1_pb^{p1}}{g^{11}} \Big)^{\frac{1}{2}} - \frac{(n\alpha-1) K^\alpha b^{11}}{(b^1_pb^{p1}g^{11})^{\frac{1}{2}}}+\alpha K^\alpha H  \Big(\frac{b^1_pb^{p1}}{g^{11}} \Big)^{\frac{1}{2}}.
 \end{align*}
Combining this with    \eqref{eq:Pre LK^a} yields
\begin{align*}\label{eq:Pog Full eq}
 \cL \, \widebar w =&-\frac{n\alpha -1}{2n\alpha} \cL \, |F|^2+ 2\Big\langle \nabla K^{\alpha},\nabla \Big(\frac{b^1_pb^{p1}}{g^{11}} \Big)^{\frac{1}{2}}\Big\rangle_{\cL}-\frac{\alpha K^{2\alpha} b^{ij}b^1_pb^1_q\nabla_i b^{1p}\nabla_j b^{1q}}{(b^1_rb^{r1})^{\frac{3}{2}}(g^{11})^{\frac{1}{2}}}  \tag{3.1}\\
 &+\frac{\alpha K^{2\alpha}b^{ij}\nabla_i b^{p1} \nabla_j b^1_p}{(b^1_qb^{q1}g^{11})^{\frac{1}{2}}}+\frac{b^1_p b^{pr}b^{1s} \nabla_r K^{\alpha} \nabla_s K^{\alpha}}{(b^1_qb^{q1}g^{11})^{\frac{1}{2}}}+\frac{\alpha K^{2\alpha} b^1_pb^{pr}b^{1s}b^{ij}b^{km}\nabla_r h_{ik}\nabla_s h_{jm}}{(b^1_qb^{q1}g^{11})^{\frac{1}{2}}}\\
 &+\Big \langle F,\nabla  \Big( K^{\alpha}\big(b^1_pb^{p1}/g^{11} \big)^{\frac{1}{2}}\Big) \Big \rangle +(n\alpha-1)  K^\alpha\Big(\frac{b^1_pb^{p1}}{g^{11}} \Big)^{\frac{1}{2}} - \frac{(n\alpha-1) K^{2\alpha} b^{11}}{(b^1_pb^{p1}g^{11})^{\frac{1}{2}}}.
 \end{align*}
Observe that 
 \begin{align*}
 2\Big\langle \nabla K^{\alpha},\nabla \Big(\frac{b^1_pb^{p1}}{g^{11}} \Big)^{\frac{1}{2}}\Big\rangle_{\cL}=2\alpha K^\alpha b^{ij}(g^{11})^{-\frac{1}{2}}\big(b^1_q b^{q1} \big)^{-\frac{1}{2}}b^1_p\nabla_i K^{\alpha} \nabla_j  b^{p1},
 \end{align*}
and 
\begin{align*}
\nabla  \Big( K^{\alpha}\big(b^1_pb^{p1}/g^{11} \big)^{\frac{1}{2}}\Big)=\nabla \widebar w +\frac{n\alpha -1}{2n\alpha}\nabla |F|^2.
\end{align*}
Hence,  applying the equations above, \eqref{eq:Pre L|F|^2} and $\nabla \widebar w(p_0)=0$ to \eqref{eq:Pog Full eq} 
yields that  the following holds at the maximum point $p_0$
\begin{align*}\label{eq:Pog 1st reduced eq}
0 \geq\, & 2\alpha K^\alpha \sum_{i=1}^n b^{ii}\nabla_i K^\alpha \nabla_i b^{11}-\alpha K^{2\alpha}\sum_{i=1}^n b^{ii}h_{11}|\nabla_i b^{11}|^2  +\alpha K^{2\alpha} \sum_{j,p} b^{jj}h_{11}|\nabla_j b^{p1}|^2 +|b^{11}\nabla_1 K^{\alpha}|^2 \tag{3.2}\\
 &+\alpha K^{2\alpha} (b^{11})^2\sum_{i,j}b^{ii}b^{jj}|\nabla_1 h_{ij}|^2+\frac{n\alpha -1}{2n\alpha} \langle F,\nabla   |F|^2 \rangle +(n\alpha-1)K^\alpha \big(b^{11}-\frac{1}{n}\sum_{i=1}^n b^{ii}\big).
 \end{align*}
By \eqref{eq:Pre Db}, the second and third terms on  the right hand side of the  inequality above \eqref{eq:Pog 1st reduced eq} satisfy
 \begin{align*}
&-\sum_{i=1}^n b^{ii}h_{11}|\nabla_i b^{11}|^2+\sum_{j,p} b^{jj}h_{11}|\nabla_j b^{p1}|^2=-\sum_{i=1}^n b^{ii}(b^{11})^3|\nabla_i h_{11}|^2+\sum_{j,p} b^{jj}b^{11}(b^{pp})^2|\nabla_j h_{p1}|^2 \\
&=\sum_{j=1}^n\sum_{p\neq 1} b^{jj}b^{11}(b^{pp})^2|\nabla_j h_{p1}|^2 \geq \sum_{p\neq 1} (b^{11} b^{pp})^2|\nabla_p h_{11}|^2 =\sum_{p\neq 1} (b^{pp}h_{11})^2|\nabla_p b^{11}|^2 .
 \end{align*}
Also, by \eqref{eq:Pre Db}  the fifth term on  the right hand side of  \eqref{eq:Pog 1st reduced eq} satisfies 
\begin{align*}
(b^{11})^2\sum_{i,j}b^{ii}b^{jj}|\nabla_1 h_{ij}|^2 \geq (b^{11})^4|\nabla_1 h_{11}|^2+2\sum_{i\neq 1}(b^{11})^3b^{ii}|\nabla_i h_{11}|^2=|\nabla_1 b^{11}|^2+2\sum_{i\neq 1}b^{ii}h_{11}|\nabla_i b^{11}|^2.
\end{align*} 
Furthermore, we have
 \begin{align*}
 \alpha K^{2\alpha} |\nabla_1 b^{11}|^2 
 +2\alpha K^\alpha b^{11}\nabla_1 K^\alpha \nabla_1 b^{11} \geq - \alpha |b^{11}\nabla_1 K^{\alpha}|^2. 
 \end{align*}
Hence, by applying the inequalities above, we can reduce \eqref{eq:Pog 1st reduced eq} to
\begin{align*}\label{eq:Pog 2nd reduced eq}
0 \geq\, & 2\alpha \sum_{i \neq 1} b^{ii}\nabla_i K^\alpha \big(K^\alpha\nabla_i b^{11}\big)+\alpha \sum_{p\neq 1} (b^{pp}h_{11})^2|K^\alpha\nabla_p b^{11}|^2 +2\alpha\sum_{i\neq 1}b^{ii}h_{11}|K^\alpha\nabla_i b^{11}|^2 \tag{3.3}\\
 &+(1- \alpha) |b^{11}\nabla_1 K^{\alpha}|^2+\frac{n\alpha -1}{2n\alpha} \langle F,\nabla   |F|^2 \rangle +(n\alpha-1)K^\alpha \big(b^{11}-\frac{1}{n}\sum_{i=1}^n b^{ii}\big).
 \end{align*} 
We now employ \eqref{eq:Pre DK^a} to obtain the following at the point $p_0$
\begin{align*}\label{eq:Prog bDK^a=D|F|^2}
b^{ii}\nabla_i K^\alpha =b^{ii}h_{ii}\langle F,F^i \rangle=\langle F,F^i \rangle.\tag{3.4}
\end{align*} 
In addition, at the point $p_0$, $\nabla_i \widebar w(p_0) =0$ yields 
\begin{align*}
K^\alpha\nabla_i b^{11}=-b^{11}\nabla_i K^\alpha +\frac{n\alpha-1}{2n\alpha}\nabla_i |F|^2=-b^{11}h_{ii}\langle F,F^i \rangle+\frac{n\alpha-1}{n\alpha}\langle F,F_i \rangle= (\beta-\theta_i)\langle F,F_i\rangle,
\end{align*} 
where $\theta_i=b^{11}h_{ii}(p_0)$ and  $\beta = \frac{n\alpha-1}{n\alpha}$. We also have 
\begin{align*}\label{eq:Pog D|F|^2}
\langle F,\nabla   |F|^2 \rangle\coloneqq \langle F, (\nabla_i   |F|^2) F^i \rangle  =\langle F, F^i\rangle (\nabla_i   |F|^2)=2\langle F,F_i\rangle\langle F,F^i\rangle.\tag{3.5}
\end{align*} 
Hence, we can rewrite \eqref{eq:Pog 2nd reduced eq} as 
  \begin{align*}\label{eq:Pog Final eq}
0 \geq\, &  \sum_{i \neq 1}\langle F,F_i \rangle^2J_i+ \langle F,F_1\rangle^2 I_1 +(n\alpha-1)K^\alpha \big(b^{11}-\frac{1}{n}\sum_{i=1}^n b^{ii}\big), \tag{3.6}
 \end{align*} 
where
 \begin{align*}
& I_1 =\frac{n\alpha -1}{n\alpha}+1-\alpha ,&&  J_i=2\alpha \Big(\beta-\theta_i\Big)+\alpha \big(\theta_i^{-2}+2\theta_i^{-1}\big)\big(\beta-\theta_i\big)^2 +\beta.
 \end{align*}
We observe that $I_1 >0 $ holds, and also $J_i$ satisfies
  \begin{align*}
 J_i=&2\alpha \beta-2\alpha \theta_i+\alpha \beta^2\theta_i^{-2}+2\alpha\beta^2\theta_i^{-1}-2\alpha\beta\theta_i^{-1}-4\alpha\beta+\alpha+2\alpha\theta_i +\beta\\
 =&\alpha(1-\beta) +\beta(1-\alpha)+2\alpha\beta(\beta-1)\theta_i^{-1}+\alpha \beta^2\theta_i^{-2}
=  \frac{1}{n}+\beta(1-\alpha)-\frac{2\beta}{n} \theta_i^{-1}+\alpha \beta^2\theta_i^{-2} \\ 
= & \beta(1-\alpha)+\frac{1}{n}+\alpha \Big(\beta \theta_i^{-1}-\frac{1}{n\alpha}\Big)^2-\frac{1}{n^2\alpha}
\geq  \beta  (1-\alpha)+\frac{1}{n}\Big(\frac{n\alpha-1}{n\alpha}\Big)=\beta(1-\alpha+\frac{1}{n})>0.
 \end{align*}
Since we have $b^{11}(p_0)=\lambda_{\min}^{-1}(p_0) \geq \lambda_i^{-1} (p_0) \geq b^{ii}(p_0)$ and $\langle F,F_i\rangle^2(p_0) \geq 0$  for all $i \in \{1,\cdots,n\}$, the inequality  \eqref{eq:Pog Final eq} and $I_1, J_i >0$ give the desired result.
\end{proof} 
 
\section{Strong maximum principle} \label{sec-Strong maximum principle}

In this section, we will show how Theorem \ref{thm:Pog Pogrelov estimate} can be modified to give us the proof of
our main result, Theorem \ref{thm:INT Uniqueness}. To this end, we 
will  introduce the new  geometric, chart-independent quantity  
$$f(p)=\big(K^{\alpha}b^{ij}g_{ij}- \frac{n\alpha-1}{2\alpha}|F|^2\big)(p)$$
and apply the  strong maximum principle.  If we use $w(p)$, $h_{ij}$ can be diagonalized only  at one given point. However, if we employ $f(p)$, we can diagonalize $h_{ij}$ at each point.  We begin with the following observation
which simply follows from Theorem \ref{thm:Pog Pogrelov estimate}. 
 
\begin{proposition}\label{prop:Smax Symmetric function}
Let $\Sigma$ be a strictly convex smooth closed solution of \eqref{eq:INT Shrinker} for an exponent $\alpha \in (\frac{1}{n},1+\frac{1}{n})$. Assume that $F:M^n \to \rno$ is a smooth immersion such that $F(M^n)= \Sigma$  , and the continuous function $ f(p)$ attains its maximum at a point $p_0$. Then, $F(p_0)$ is an umbilical point and $\nabla |F|^2(p_0)=0$ holds.
\end{proposition}

\begin{proof}
We observe $b^{ij}g_{ij}(p)=\sum^n_{i=1} \lambda_i^{-1}(p)$, where $\lambda_1(p),\cdots,\lambda_n(p)$ are the principal curvatures of $\Sigma$ at $F(p)$. Therefore, we have $f(p) \leq n \, w(p)$. However, if $w(p_0)=\max_{p\in M^n} w(p)$, then $f(p_0)=n\, w(p_0)$ holds, because $F(p_0)$ is an umbilical 
point by Theorem \ref{thm:Pog Pogrelov estimate}. Hence, we have 
\begin{align*}
f(p) \leq n\, w(p)\leq  \max_{p \in M^n}n\, w (p)= \max_{p \in M^n}f(p).
\end{align*}
Thus, if $f$ attains its maximum at a point $p_0$, then $w$ also attains its maximum at $p_0$, and thus we can obtain the desired result by Theorem \ref{thm:Pog Pogrelov estimate}.
\end{proof}

\bigskip

We will now employ the strong maximum principle to  prove Theorem \ref{thm:INT Uniqueness}.
 
\begin{proof}[Proof of Theorem \ref{thm:INT Uniqueness}]
We define a set $M_f \subset M^n$ by 
$$M_f=\{p\in M^n: f(p)=\max_{M^n}f\}.$$  Since $f(p)$ is a continuous function defined on a closed manifold $M^n$, $f$ attains its maximum, and thus $M_f$ is not an empty set. We now define the continuous function $\Lambda:M^n \to \mathbb{R}$ and the open set $V\subset M^n$ by
\begin{align*}
 \Lambda(p)=\sum_{i,j}\Big(\frac{\lambda_i}{\lambda_j}-\frac{\lambda_j}{\lambda_i}\Big)^2(p), \quad \qquad  V=\Big\{p \in M^n:\Lambda(p)< \Big(\frac{10}{9}-\frac{9}{10}\Big)^2 \Big\}.
\end{align*}

We now begin by combining \eqref{eq:Pre LK^a} and \eqref{eq:Pre Lb} to obtain
\begin{align*}
\cL \, \big (K^\alpha b^{pq} \big ) =&2\langle \nabla K^\alpha ,\nabla b^{pq} \rangle_{\cL}+ b^{pr}b^{qs}\nabla_r K^{\alpha} \nabla_s K^\alpha+\alpha K^{2\alpha} b^{pr}b^{qs}b^{ij}b^{km}\nabla_r h_{ik}\nabla_s h_{jm}\\
 &+\langle F,\nabla (K^\alpha b^{pq}) \rangle +(n\alpha -1)K^\alpha (b^{pq} -g^{pq} K^{\alpha}).
\end{align*}
Therefore, we can derive the following from \eqref{eq:Pre L|F|^2} and $\nabla g_{pq}=0$
\begin{align*}
\cL \, f=&2g_{pq}\langle \nabla K^\alpha ,\nabla b^{pq} \rangle_{\cL}+ b^{pr}b^{s}_p\nabla_r K^{\alpha} \nabla_s K^\alpha+\alpha K^{2\alpha} b^{pr}b^{s}_pb^{ij}b^{km}\nabla_r h_{ik}\nabla_s h_{jm}+\langle F,\nabla (K^\alpha b^{pq}g_{pq}) \rangle.
\end{align*}
By using \eqref{eq:Pog D|F|^2}, we can obtain
\begin{align*}
\langle F,\nabla (K^\alpha b^{pq}g_{pq}) \rangle=\langle F,\nabla f \rangle +\frac{n\alpha -1}{2\alpha}\langle F,\nabla |F|^2\rangle=\langle F,\nabla f \rangle +\big( n -\alpha^{-1} \big)\langle F, F_i\rangle\langle F,F^i\rangle.
\end{align*}
Hence, we have
\begin{align*}\label{eq:Smax General eq}
\cL \, f-\langle F,\nabla f \rangle =  &2\alpha ( b^{ij}\nabla_i K^\alpha) (K^\alpha g_{pq}\nabla_j b^{pq})+ b^{pr}b^{s}_p\nabla_r K^{\alpha} \nabla_s K^\alpha \tag{4.1}\\
&+\alpha K^{2\alpha} b^{pr}b^{s}_pb^{ij}b^{km}\nabla_r h_{ik}\nabla_s h_{jm}+\big( n -\alpha^{-1} \big)\langle F, F_i\rangle\langle F,F^i\rangle.
\end{align*}
Given a fixed point $p_0 \in V$, we choose an orthonormal frame at $F(p_0)$ satisfying \begin{align*}
g_{ij}(p_0)=\delta_{ij}, \qquad  \quad  h_{ij}(p_0)=\lambda_{i}(p_0)\delta_{ij}.
\end{align*}
Then, at the point $p_0$, we can rewrite \eqref{eq:Smax General eq} as 
\begin{align*}\label{eq:Smax General eq on orthonormal frame}
\cL \, f-\langle F,\nabla f \rangle =  &2\alpha \sum_{i,j}( b^{ii}\nabla_i K^\alpha) (K^\alpha \nabla_i b^{jj})+ \sum_{i}|b^{ii}\nabla_i K^{\alpha}|^2 \tag{4.2}\\
&+\alpha K^{2\alpha}\sum_{i,j,k} (b^{ii})^2b^{jj}b^{kk}|\nabla_i h_{jk}|^2+\big( n -\alpha^{-1} \big)\sum_i\langle F, F_i\rangle^2.
\end{align*}
Since $p_0 \in V$ and  the definition of $V$ guarantees that $b^{ii}h_{jj}(p_0) \geq \frac{9}{10}$,  
by using \eqref{eq:Pre Db} we can derive
\begin{align*}
\alpha K^{2\alpha}\sum_{i,j,k} (b^{ii})^2b^{jj}b^{kk}|\nabla_i h_{jk}|^2 \geq & \alpha\sum_i |K^\alpha\nabla_i b^{ii}|^2+2\alpha\sum_{i\neq j}b^{jj}h_{ii}|K^\alpha\nabla_j b^{ii}|^2+\alpha\sum_{i \neq j}(b^{ii}h_{jj})^2|K^\alpha\nabla_i b^{jj}|^2\\
\geq &\alpha\sum_i |K^\alpha\nabla_i b^{ii}|^2+\frac{5}{2} \alpha\sum_{i \neq j}|K^\alpha\nabla_i b^{jj}|^2.
\end{align*}
We also have
\begin{align*}
\alpha\sum_i |K^\alpha\nabla_i b^{ii}|^2 +2\alpha \sum_{i}( b^{ii}\nabla_i K^\alpha) (K^\alpha \nabla_i b^{ii})\geq &-\alpha\sum_{i}|b^{ii}\nabla_i K^{\alpha}|^2
\end{align*}
and
\begin{align*}
\frac{5}{2} \alpha\sum_{i \neq j}|K^\alpha\nabla_i b^{jj}|^2+2\alpha \sum_{i\neq j}( b^{ii}\nabla_i K^\alpha) (K^\alpha \nabla_i b^{jj}) \geq & -\frac{2}{5}\alpha\sum_{i\neq j}| b^{ii}\nabla_i K^\alpha|^2=-\frac{2}{5}\alpha(n-1)\sum_{i}| b^{ii}\nabla_i K^\alpha|^2.
\end{align*}
Applying the inequalities above and \eqref{eq:Prog bDK^a=D|F|^2} to \eqref{eq:Smax General eq on orthonormal frame} yields
\begin{align*}
\cL \, f-\langle F,\nabla f \rangle \geq &\Big((1-\alpha)-\frac{2}{5}\alpha(n-1)+(n-\alpha^{-1}) \Big)\sum_i\langle F, F_i\rangle^2\\
=&\frac{1}{5\alpha}\Big(-(2n+3)\alpha^2+5(n+1)\alpha-5 \Big)\sum_i\langle F, F_i\rangle^2.
\end{align*}
Let us consider the function $y(\alpha)=-(2n+3)\alpha^2+5(n+1)\alpha-5$. Then, we have
\begin{align*}
&y(1+1/n)=3n-2-(3/n)-(3/n^2) \geq 0 , && y(1/n)=(3/n)-(3/n^2) \geq 0,
\end{align*}
which implies $y(\alpha) \geq 0$ for $\alpha \in [\frac{1}{n},1+\frac{1}{n}]$. Therefore, on $V$ the following holds
\begin{align*}
\cL \, f-\langle F,\nabla f \rangle \geq 0.
\end{align*}
Notice that $\cL \, f-\langle F,\nabla f \rangle$ is a chart-independent function. Hence, the Hopf maximum principle and 
$M_f \subset V$ show that $M_f=V$. However, $M_f$ is a closed set and $V$ is an open set by the continuity  of $f$ and 
$\Lambda$, respectively. So, we conclude that  $M_f=M^n$, and thus Proposition \ref{prop:Smax Symmetric function} gives the desired result.
\end{proof}

\bibliographystyle{abbrv}

\bibliography{myref}

\centerline{\bf Acknowledgements}

\smallskip

\noindent P. Daskalopoulos and K. Choi have been partially supported by NSF grant DMS-1600658.

\end{document}